\documentclass[graybox]{svmult}

\usepackage{mathptmx}       
\usepackage{helvet}         
\usepackage{courier}        
\usepackage{type1cm}        
\usepackage{makeidx}         
\usepackage{graphicx}        
\usepackage{multicol}        
\usepackage[bottom]{footmisc}

\makeindex             

\usepackage{amsmath}
\usepackage{color}
\usepackage{amssymb}
\usepackage{latexsym}
\usepackage{url}
\usepackage{multirow}

\def\d{\displaystyle}
\def\p{\partial}

\def\e{\varepsilon}

\def\R{{\bf R}}
\def\N{{\bf N}}

\def\wt{\widetilde}
\def\ints{\int_{\R^n}}

\def\l{\lambda}
\newcommand{\norm}[1]{\left\lVert#1\right\rVert}
\DeclareMathOperator*{\supp}{supp}

\begin{document}

\title*{Wave-like blow-up for semilinear wave equations
with scattering damping and negative mass term}
\titlerunning{Wave-like blow-up for semilinear wave equations}
\author{Ning-An Lai, Nico Michele Schiavone and Hiroyuki Takamura}
\institute{Ning-An Lai \at Institute of Nonlinear Analysis and Department of Mathematics, Lishui University, Lishui City 323000, China, \email{ninganlai@lsu.edu.cn}
\and Nico Michele Schiavone \at M.Sc. student, Department of Mathematics, University of Bari, Via E. Orabona 4, Bari 70126, Italy, \email{n.schiavone5@studenti.uniba.it}
\and Hiroyuki Takamura \at Mathematical Institute, Tohoku University, Aoba, Sendai 980-8578, Japan, \email{hiroyuki.takamura.a1@tohoku.ac.jp}}
%
%
\maketitle

\abstract*{In this paper we establish blow-up results and lifespan estimates for semilinear wave equations with scattering damping and negative mass term for subcritical power, which is the same as that of the corresponding problem without mass term, and also the same as that of the corresponding problem without both damping and mass term. For this purpose, we have to use the comparison argument twice, due to the damping and mass term, in additional to a key multiplier. Finally, we get the desired results by an iteration argument.}

\abstract{In this paper we establish blow-up results and lifespan estimates for semilinear wave equations with scattering damping and
negative mass term for subcritical power, which are the same as that of the corresponding problem without mass term, and also the same as that of the corresponding problem without both damping and mass term. For this purpose, we have to use the comparison argument twice, due to the damping and mass term, in additional to a key multiplier. Finally, we get the desired results by an iteration argument.}

\section{Introduction}
In this paper, we consider the Cauchy problem for semilinear wave equations with scattering damping and negative mass term
\begin{equation}
\label{eq:main_problem}
\left\{
\begin{aligned}
& u_{tt} - \Delta u + \frac{\mu_1}{(1+t)^\beta} u_t - \frac{\mu_2}{(1+t)^{\alpha+1}} u = |u|^p, \quad\text{in $\R^n\times[0, T)$}, \\
& u(x,0)=\e f(x), \quad u_t(x,0)=\e g(x), \quad x\in\R^n,
\end{aligned}
\right.
\end{equation}
where  $\mu_1, \mu_2 > 0$, $\alpha>1$, $\beta >1$, $n\in\N$ and $\e>0$ is a ``small'' parameter.
\par
We call the term $\mu_1u_t/(1+t)^\beta\,(\beta>1)$ scattering damping, due to the reason that the solution of the following Cauchy problem
\begin{equation}
\label{linear}
\left\{
\begin{aligned}
& u^0_{tt}-\Delta u^0+\frac{\mu}{(1+t)^\beta}u^0_t=0,
\quad \text{in $\R^n \times[0,\infty)$},\\
& u^0(x,0)=u_1(x), \quad u^0_t(x,0)=u_2(x), \quad x\in\R^n,
\end{aligned}
\right.
\end{equation}
scatters to that of the free wave equation when $\beta>1$ and $t\rightarrow \infty$. In fact, according to the works of Wirth \cite{Wir1, Wir2, Wir3}, we may classify the damping for different values of $\beta$ into four cases, as shown in the next table.

\begin{table}
\centering
	\begin{tabular}{p{3cm}p{4.5cm}}
		\hline\noalign{\smallskip}
		Range of $\beta$ & Classification  \\
		\noalign{\smallskip}\svhline\noalign{\smallskip}
		$\beta\in(-\infty,-1)$ & overdamping\\
		$\beta\in[-1,1)$ & effective\\
		\multirow{2}{*}{$\beta=1$} & scaling invariant\\
		& \quad if $\mu\in(0,1)\Rightarrow$ non-effective\\
		$\beta\in(1,\infty)$ & scattering\\
		\noalign{\smallskip}\hline\noalign{\smallskip}
	\end{tabular}
\end{table}

If we come to the nonlinear problem with power nonlinearity, thus
\begin{equation}
\label{nonlinear}
\left\{
\begin{aligned}
& u_{tt}-\Delta u+\frac{\mu}{(1+t)^\beta}u_t=|u|^p,
\quad \text{in $R^n \times[0,\infty)$},\\
& u(x,0)=u_1(x), \quad u_t(x,0)=u_2(x), \quad x\in\R^n,
\end{aligned}
\right.
\end{equation}
we want to determine the long time behaviour of the solution according to the different value of $p$, $n$ and even $\mu$. Ikeda and Wakasugi \cite{IWnew} proved global existence for \eqref{nonlinear} for all $p>1$ when $\beta<-1$. For $\beta\in [-1, 1)$, due to the work \cite{DLR13, LNZ12, N1, WY17, IO, FIW, II}, we know that problem \eqref{nonlinear} admits a critical power $p_F(n):=1+2/n$ (Fujita power), which means that
for $p\in (1, p_F(n)]$ the solution will blow up in a finite time, while for $p\in \left(p_F(n), \infty\right)$ we have global existence. Obviously, in this case
the critical is exactly the same as that of the Cauchy problem of semilinear heat equation
\[
u_t-\Delta u=u^p,
\]
and so we call it admits \lq\lq heat-like" behaviour.

For the case $\beta=1$ in \eqref{nonlinear}, we say that the damping is scale invariant, due to the reason that the equation in the corresponding linear problem \eqref{linear} is invariant under the following scaling transformation
\[
\wt{u^0}(x,t):=u^0(\sigma x, \sigma(1+t)-1),\ \sigma>0.
\]
It is a bit sophisticated for the scale invariant nonlinear problem \eqref{nonlinear}, since the size of the positive constant $\mu$ will also
have an effect on the long time behaviour of the solution. Generally speaking, according to the known results (\cite{DABI, DL1, DLR15, WY14_scale, LTW, IS, TL1, TL2}), it is believed that if $\mu$ is large enough, then the critical power is related to the Fujita power, while
if $\mu$ is relatively small, then the critical power is related to the Strauss power, i.e. $p_S(n)$, which is denoted to be the positive root of the following quadratic equation
\[
\gamma(p,n):=2+(n+1)p-(n-1)p^2=0,
\]
and which is also the critical power of the small data Cauchy problem of the semilinear wave equation
\[
u_{tt}-\Delta u=|u|^p.
\]
It means that for relatively small $\mu$ we have \lq\lq wave-like" behaviour. Unfortunately, we are not clear of the exact threshold determined by the value $\mu$ between the \lq\lq heat-like" and \lq\lq wave-like"
phenomenon till now.

For the scattering case $(\beta>1)$, one expects that problem \eqref{nonlinear} admits the long time behaviour as that of the corresponding problem
without damping. In \cite{LT}, Lai and Takamura obtained the blow-up results for
\[
1<p<
\begin{cases}
p_S(n) &\text{for $n\ge2$},\\
\infty &\text{for $n=1$}
\end{cases}
\]
and the upper bound of the lifespan estimate
\[
T\leq C\varepsilon^{-2p(p-1)/\gamma(p, n)}.
\]
What is more, when $n=1, 2$ and
\[
\int_{\R^n}g(x)dx\neq 0,
\]
they established an improved upper bound of the lifespan for $1<p<2$, $n=2$ and $p>1$, $n=1$. However, it remains to determine the exact critical power for \eqref{nonlinear} with $\beta>1$.

Recently, the small data Cauchy problem for semilinear wave equation with scale-invariant damping and mass and power non-linearity, i.e.,
\begin{equation}
\label{nonlinearmass}
\left\{
\begin{aligned}
& u_{tt} - \Delta u + \frac{\mu_1}{1+t} u_t + \frac{\mu_2^2}{(1+t)^{2}} u = |u|^p,
\quad \text{in $\R^n \times[0,\infty)$},\\
& u(x,0)=u_1(x), \quad u_t(x,0)=u_2(x), \quad x\in\R^n,
\end{aligned}
\right.
\end{equation}
attracts more and more attention. Denote
\begin{equation}
\label{eq:def_delta}
\delta:=(\mu_1-1)^2-4\mu_2^2.
\end{equation}
Then in \cite{NPR} and \cite{Pal} a blow-up result was established for
\begin{equation*}
1<p\leq p_{F}\left(n+\frac{\mu_1-1-\sqrt{\delta}}{2}\right)
\end{equation*}
assuming $\delta\geq 0$, by using two different approaches. Furthermore, in \cite{NPR} they improved the result for $\delta=1$ to
\[
1<p\leq \max\left\{p_S(n+\mu_1),\ p_F\left(n+\frac{\mu_1}{2}-1\right)\right\}.
\]
Recently, Palmieri and Reissig \cite{P-R} generalized the blow-up result for $n\geq 1$ and $\delta\in (0, 1]$ to the following power:
\[
\begin{cases}
p<p_{\mu_1, \mu_2}(n):=\max\left\{p_S(n+\mu_1),\ p_F\left(n+\frac{\mu_1}{2}-\frac{\sqrt{\delta}}{2}\right)\right\},\\
p=p_{\mu_1, \mu_2}(n)=p_F\left(n+\frac{\mu_1}{2}-\frac{\sqrt{\delta}}{2}\right),\\
p=p_{\mu_1, \mu_2}(n)=p_S(n+\mu_1), \quad\text{for $n=2$}.
\end{cases}
\]
We note that a transform by $v:=(1+t)^{\mu_1/2}u$ changes the equation in (\ref{nonlinearmass})
into
\[
v_{tt}-\Delta v+\frac{1-\delta}{4(1+t)^2}v=\frac{|v|^p}{(1+t)^{\mu_1(p-1)/2}},
\]
so that the assumption of $\delta\in(0,1]$ implies the non-negativeness of the mass term in this equation.

In this paper, we are going to study the small data Cauchy problem of semilinear wave equations with power nonlinearity, scattering damping and mass term with negative sign, thus, problem \eqref{eq:main_problem}. Blow-up results and lifespan estimates will be established for $1<p<p_S(n)$, which are the same as that in the work \cite{LT}. We could say that we experience a double phenomenon of scattering, due to the damping term and the mass term.
For the proof, we will borrow the idea from \cite{LT}, by introducing a key multiplier to absorb the damping term and establishing an iteration frame. However, we have to deal with the mass term. Due to the negative sign, we use a comparison argument to eliminate the effect from the mass term.
Although the calculations in this work hold for any mass exponent $\alpha\in\R$, we suppose that it satisfies $\alpha>1$ because otherwise we have shorter lifespan estimates due to the effect of the negative mass term. This analysis will appear in our forthcoming paper.

\section{Main Result}
\par\quad
Before the statement of our main results, we first denote the energy and weak solutions of problem \eqref{eq:main_problem}.
\begin{definition}
	We say that $u$ is an energy solution of \eqref{eq:main_problem} over $[0,T)$ if
	\begin{equation}
	u \in C([0,T), H^1(\R^n)) \cap C^1([0,T),L^2(\R^n)) \cap C((0,T), L^p_{loc}(\R^n))
	\end{equation}
	satisfies $u(x,0)=\e f(x)$ in $H^1(\R^n)$ and $u_t(x,0)=\e g(x)$ in $L^2(\R^n)$, and
	\begin{equation}
	\label{eq:energy_solution}
	\begin{split}
	&\int_{\R^n}u_t(x,t)\phi(x,t)dx-\int_{\R^n}\e g(x)\phi(x,0)dx\\
	&+\int_0^tds\int_{\R^n}\left\{-u_t(x,s)\phi_t(x,s)+\nabla u(x,s)\cdot\nabla\phi(x,s)\right\}dx \\
	&+\int_0^tds\int_{\R^n}\frac{\mu_1 }{(1+s)^{\beta}}u_t(x,s) \phi(x,s)dx - \int_0^t ds \ints \frac{\mu_2}{(1+s)^{\alpha+1}}u(x,s)\phi(x,s) \\
	= & \int_0^tds\int_{\R^n}|u(x,s)|^p\phi(x,s)dx
	\end{split}
	\end{equation}
	with any test function $\phi \in C_0^\infty(\R^n \times [0,T))$ and for any $t\in[0,T)$.
\end{definition}

Employing the integration by part in the above equality and letting $t\to T$, we got the definition of the weak solution of \eqref{eq:main_problem}, that is
\begin{equation}
\begin{split}
&\int_{\R^n\times[0,T)}
u(x,s)\bigg\{\phi_{tt}(x,s)-\Delta \phi(x,s)
-\frac{\p}{\p s} \left(\frac{\mu_1}{(1+s)^{\beta}}\phi(x,s)\right)\\
& - \frac{\mu_2}{(1+s)^{\alpha+1}} \phi(x,s) \bigg\}dxds\\
=&\ \int_{\R^n}\mu_1 \e f(x)\phi(x,0)dx -\int_{\R^n}\e f(x)\phi_t(x,0)dx +\int_{\R^n}\e g(x)\phi(x,0)dx\\ &+\int_{\R^n\times[0,T)}|u(x,s)|^p\phi(x,s)dxds.
\end{split}
\end{equation}

\begin{definition}
	As in the introduction, set
	\[
	\gamma(p,n) := 2+(n+1)p - (n-1)p^2
	\]
	and, for $n\ge2$, define $p_S(n)$ the positive root of the quadratic equation $\gamma(p,n)=0$, the so-called Strauss exponent, that is
	\[
	p_S(n) = \frac{n+1+\sqrt{n^2+10n-7}}{2(n-1)}.
	\]
	Note that if $n=1$, then $\gamma(p,1)=2+2p$ and we can set $p_S(1):=+\infty$.
\end{definition}
\par
Now we announce our main results.
\begin{theorem}\label{thm1}
	Let $n=1$ and $p>1$, or $n \geq 2$ and $1<p<p_S(n)$. Assume that both $f \in H^1(R^n)$ and $g\in L^2(\R^n)$ are non-negative, and at least one of them does not vanish identically. Suppose that $u$ is an energy solution of \eqref{eq:main_problem} on $[0,T)$ that satisfies
	\begin{equation}
	\label{support}
	\supp u \subset\{(x,t)\in\R^n\times[0,\infty) \colon |x|\le t+R\}
	\end{equation}
	with some $R\ge1$.
	Then, there exists a constant $\e_0=\e_0(f,g,n,p,\mu_1, \beta, R)>0$ which is independent of $\mu_2$,
	such that $T$ has to satisfy
	\begin{equation}
	\label{eq:1}
	T\leq C\varepsilon^{-2p(p-1)/\gamma(p, n)}
	\end{equation}
	for $0<\e\le\e_0$, where $C$ is a positive constant independent of $\e$.
\end{theorem}

In low dimensions $(n=1, 2)$, with some additional hypothesis, we may have improvements on the lifespan estimates as follows.

\begin{theorem}
	\label{thm2}
	Let $n=2$ and $1<p<2$.
	Assume that both $f \in H^1(R^2)$ and $g\in L^2(\R^2)$ are non-negative and that $g$ does not vanish identically. Then the lifespan estimate \eqref{eq:1} is replaced by
	\begin{equation}
	\label{1a}
	T\leq C\varepsilon^{-(p-1)/(3-p)}.
	\end{equation}
\end{theorem}

\begin{theorem}\label{thm3}
	Let $n=1$ and $p>1$.
	Assume that both $f \in H^1(R^1)$ and $g\in L^2(\R^1)$ are non-negative and that $g$ does not vanish identically. Then the lifespan estimate \eqref{eq:1} is replaced by
	\begin{equation}
	\label{1b}
	T\leq C\varepsilon^{-(p-1)/2}.
	\end{equation}
\end{theorem}

\begin{theorem}\label{thm4}
	Let $n=p=2$. Suppose that $\alpha\le\beta$ and
	\begin{equation}
	\label{eq:cond_thm4}
	\mu_2 \ge
	\begin{cases}
	\frac{\beta\mu_1}{2} & \text{if $\alpha=\beta$}, \\
	\frac{\beta\mu_1}{2}\frac{\beta-1}{2\beta-\alpha-1}\left(4\frac{\mu_1^2}{\mu_2}\frac{\beta-\alpha}{\beta-1} \right)^{\frac{\beta-\alpha}{2\beta-\alpha-1}} & \text{if $\alpha<\beta$}.
	\end{cases}
	\end{equation}
	Assume that $f \equiv 0$ and $g \in C^2(\R^2)$ is non-negative and does not vanish identically. Suppose also that $u$ is a classical solution of \eqref{eq:main_problem} on $[0,T)$ with the support property \eqref{support}. Then, $T$ satisfies
	\begin{equation}
	\label{1c}
	T \leq C a(\e)
	\end{equation}
	where $a=a(\e)$ is a number satisfying
	\begin{equation}
	\label{def:a}
	a^2 \e^2 \log(1+a)=1.
	\end{equation}
\end{theorem}

\begin{remark}
	In Theorem \ref{thm1}, we require that at least one of the initial data does not vanish identically, which is weaker than that in the corresponding result (Theorem 2.1) in \cite{LT}.
\end{remark}

\begin{remark}
	Observe that:
	\begin{itemize}
		\item \eqref{1a} is stronger than \eqref{eq:1} by the fact that
		$1<p<2$ is equivalent to
		\[
		\frac{p-1}{3-p}<\frac{2p(p-1)}{\gamma(p,2)};
		\]
		
		\item \eqref{1b} is stronger than \eqref{eq:1} by the fact that
		$p>1$ is equivalent to
		\[
		\frac{p-1}{2}<\frac{2p(p-1)}{\gamma(p,1)};
		\]
		
		\item \eqref{1c} is stronger than \eqref{eq:1} by the fact that when $n=p=2$
		\[
		a(\e) < \e^{-1} =\e^{-2\cdot2(2-1)/\gamma(2,2)}
		\]
		for sufficiently small $\e$.
	\end{itemize}
\end{remark}

\section{Lower bound for derivative of the functional}

Following the idea in \cite{LT}, we introduce the multiplier
\begin{equation}
\label{eq:m}
m(t):= \exp\left(\mu_1 \frac{(1+t)^{1-\beta}}{1-\beta}\right).
\end{equation}
Clearly
\begin{equation}
\label{eq:m_bound}
1 \geq m(t) \geq m(0) >0 \quad\text{for $t \geq 0$}.
\end{equation}
Moreover, let us define the functional
\begin{equation*}
F_0(t) := \int_{\R^n} u(x,t)dx,
\end{equation*}
then
\[
F_0(0) = \e\int_{\R^n} f(x)dx, \quad F'_0(0) = \e\int_{\R^n}g(x)dx
\]
are non-negative due to the hypothesis of positiveness on the initial data.
Our final target is to establish a lower bound for $F_0(t)$.

Let us start finding the lower bound of the derivative of the functional, i.e., $F'_0(t)$. 
Due to \eqref{support}, choosing the test function $\phi=\phi(x,s)$ in \eqref{eq:energy_solution} to satisfy
$\phi\equiv 1$ in $\{(x,s)\in \R^n\times[0,t]:|x|\le s+R\}$, we get
\[
\begin{split}
&\int_{\R^n}u_t(x,t)dx-\int_{\R^n}u_t(x,0)dx+\int_0^tds\int_{\R^n}\frac{\mu_1}{(1+s)^\beta}u_t(x,s)dx \\
=& \int_0^t \int_{\R^n} \frac{\mu_2}{(1+s)^{\alpha+1}} u(x,s) dx
+\int_0^tds\int_{\R^n}|u(x,s)|^pdx,
\end{split}
\]
which yields by taking derivative with respect to $t$
\begin{equation}
\label{2}
F_0''(t)+\frac{\mu_1}{(1+t)^\beta}F_0'(t)
= \frac{\mu_2}{(1+t)^{\alpha+1}}F_0(t) + \int_{\R^n}|u(x,t)|^pdx.
\end{equation}
Here we note that \eqref{2} can be established by regularity assumption on the solution.
Multiplying both sides of \eqref{2} with $m(t)$ yields
\begin{equation}
\label{3}
\left\{m(t)F'_0(t)\right\}'
= m(t)\frac{\mu_2}{(1+t)^{\alpha+1}} \, F_0(t) + m(t)\int_{\R^n}|u(x,t)|^pdx.
\end{equation}
Integrating the above equality over $[0, t]$ we get
\begin{equation}
\label{4}
\begin{split}
F'_0(t) =&\ \frac{m(0)}{m(t)}F'_0(0) + \frac{1}{m(t)} \int_0^t m(s)\frac{\mu_2}{(1+s)^{\alpha+1}}F_0(s) ds\\
&+ \frac{1}{m(t)} \int_0^t m(s)ds\int_{\R^n}|u(x,s)|^pdx.
\end{split}
\end{equation}

To get the lower bound for $F'_0$, we need the positiveness of $F_0$, and this can be obtained by a comparison argument. However, since we assume that at least one of the initial data does not vanish identically, we have to consider the following two cases.

\textbf{Case 1: $f\geq 0 (\not\equiv 0)$, $g \geq 0$.} This means that $F_0(0)>0$, $F'_0(0) \geq 0$. By the continuity of $F_0$, it is positive at least for small time. Suppose that $t_0$ is the smallest zero point of $F_0$, such that $F_0>0$ in $[0,t_0)$. Then, integrating \eqref{4} over this interval we have
\begin{equation*}
\begin{split}
0 = F_0(t_0) =&\ F_0(0)  + m(0)F'_0(0) \int_0^{t_0} \frac{ds}{m(s)} \\
&+ \int_{0}^{t_0} \frac{ds}{m(s)} \int_0^s m(r)\frac{\mu_2}{(1+r)^{\alpha+1}}F_0(r) dr\\
&+ \int_{0}^{t_0}\frac{ds}{m(s)} \int_0^s m(r)dr\int_{\R^n}|u(x,r)|^pdx >0,
\end{split}
\end{equation*}
which leads to a contradiction, and hence $F(t)$ is positive all the time.

\textbf{Case 2: $f\geq 0$, $g \geq 0 (\not\equiv 0)$.} This imply that $F_0(0) \geq 0$, $F'_0(0)>0$. We apply the same argument as in the first case to $F'_0$. Suppose that $t_0$ is the smallest zero point of $F'_0$, such that $F'_0$ is positive on the interval $[0,t_0)$. Therefore $F_0$ is
strictly monotone increasing on the same interval, and hence positive due to $F_0(0) \geq 0$. Letting $t=t_0$ in \eqref{4}, we again come to a contradiction. Therefore $F'_0$ is always strictly positive, and hence $F_0(t)>0$ holds for all $t>0$.

Coming back to \eqref{4}, using the positivity of $F_0$, the boundedness of $m(t)$ and that $F'_0(0) \geq 0$, we obtain the lower bound for $F'_0$ as
\begin{equation}
\label{eq:bound_F'_0}
F'_0(t) \geq m(0) \int_0^t \int_{\R^n} |u(x,s)|^p dxds \quad\text{for $t \geq 0$}.
\end{equation}

\section{Lower bound for the weighted functional}
Set
\begin{equation*}
F_1(t):=\int_{\R^n}u(x,t)\psi_1(x,t)dx,
\end{equation*}
where $\psi_1$ is the test function introduced by Yordanov and Zhang \cite{YZ06}
\begin{equation*}
\psi_1(x,t):=e^{-t}\phi_1(x),
\quad
\phi_1(x):=
\begin{cases}
\d\int_{S^{n-1}}e^{x\cdot\omega}dS_\omega & \text{for $n\ge2$},\\
e^x+e^{-x} & \text{for $n=1$}.
\end{cases}
\end{equation*}

\begin{lemma}[Inequality (2.5) of Yordanov and Zhang \cite{YZ06}]\label{lem1}
	\begin{equation}
	\label{14}
	\int_{|x|\leq t+R}\left[\psi_1(x,t)\right]^{p/(p-1)}dx
	\leq C(1+t)^{(n-1)\{1-p/(2(p-1))\}},
	\end{equation}
	where $C_1=C_1(n,p,R)>0$.
\end{lemma}

Next we aim to establish the lower bound for $F_1$. From the definition of energy solution \eqref{eq:energy_solution}, we have that
\[
\begin{split}
&\frac{d}{dt}\int_{\R^n}u_t(x,t)\phi(x,t)dx +\int_{\R^n}\left\{-u_t(x,t)\phi_t(x,t)-u(x,t)\Delta\phi(x,t)\right\}dx\\
&+\int_{\R^n}\frac{\mu_1 }{(1+t)^\beta}u_t(x,t)\phi(x,t)dx
- \int_{\R^n} \frac{\mu_2}{(1+t)^{\alpha+1}}u(x,t)\phi(x,t)dx\\
=&\ \int_{\R^n}|u(x,t)|^p\phi(x,t)dx.
\end{split}
\]
Multiplying both sides of the above equality with $m(t)$ yields
\[
\begin{split}
&\frac{d}{dt}\left\{m(t)
\int_{\R^n}u_t(x,t)\phi(x,t)dx\right\}\\
&+m(t)\int_{\R^n}\left\{-u_t(x,t)\phi_t(x,t)-u(x,t)\Delta\phi(x,t)\right\}dx\\
=&\ m(t)\int_{\R^n} \frac{\mu_2}{(1+t)^{\alpha+1}}u(x,t)\phi(x,t)dx + m(t)\int_{\R^n}|u(x,t)|^p\phi(x,t)dx,
\end{split}
\]
integrating which over $[0,t]$ yields
\[
\begin{split}
&m(t)
\int_{\R^n}u_t(x,t)\phi(x,t)dx
-m(0)\e\int_{\R^n}g(x)\phi(x,0)dx\\
&-\int_0^tds\int_{\R^n}m(s)u_t(x,s)\phi_t(x,s)dx - \int_0^tds\int_{\R^n} m(s)
u(x,s)\Delta\phi(x,s) \\
=&\ \int_0^tds\int_{\R^n}m(s) \frac{\mu_2}{(1+s)^{\alpha+1}}u(x,s)\phi(x,s)dx\\
&+ \int_0^tds\int_{\R^n} m(s)|u(x,s)|^p\phi(x,s)dx.
\end{split}
\]
Integrating by parts the first term in the second line of the above equality, we have
\begin{equation}
\label{eqforF_1}
\begin{split}
&m(t)
\int_{\R^n}u_t(x,t)\phi(x,t)dx
-m(0)\e\int_{\R^n}g(x)\phi(x,0)dx\\
&-m(t)\int_{\R^n} u(x,t)\phi_t(x,t)dx + m(0)\e \int_{\R^n} f(x)\phi_t(x,0)dx\\
&+ \int_0^t ds \int_{\R^n} m(s)\frac{\mu_1}{(1+s)^\beta}u(x,s)\phi_t(x,s)dx
\\
& + \int_0^t ds \int_{\R^n} m(s)u(x,s)\phi_{tt}(x,s)dx - \int_0^tds\int_{\R^n} m(s)
u(x,s)\Delta\phi(x,s) \\
=&\ \int_0^tds\int_{\R^n}m(s) \frac{\mu_2}{(1+s)^{\alpha+1}}u(x,s)\phi(x,s)dx\\
&+ \int_0^tds\int_{\R^n} m(s)|u(x,s)|^p\phi(x,s)dx.
\end{split}
\end{equation}

Setting
\[
\phi(x,t)=\psi_1(x,t)=e^{-t}\phi_1(x)
\quad\text{on $\supp u$},
\]
then we have
\[
\phi_t=-\phi,\quad \phi_{tt}=\Delta\phi \quad\text{on $\supp u$}.
\]
Hence we obtain from \eqref{eqforF_1}
\[
\begin{split}
m(t)\{F_1'(t)+2F_1(t)\}
=& \ m(0)\e\int_{\R^n}\left\{ f(x)+g(x)\right\}\phi_1(x)dx \\
&+ \int_0^t m(s) \left\{\frac{\mu_1}{(1+s)^\beta} + \frac{\mu_2}{(1+s)^{\alpha+1}} \right\} F_1(s)ds\\
&+\int_0^tds\int_{\R^n}m(s)|u(x,s)|^pdx,
\end{split}
\]
which implies
\begin{equation}
\label{eq:F'_1+2F_1}
\begin{split}
F'_1(t)+2F_1(t) &\ge\frac{m(0)}{m(t)}
C_{f,g}\e+\frac{1}{m(t)}\int_0^tm(s)
\left\{\frac{\mu_1}{(1+s)^\beta} + \frac{\mu_2}{(1+s)^{\alpha+1}} \right\}F_1(s)ds\\
&\ge m(0)C_{f,g}\e+\int_0^tm(s)
\left\{\frac{\mu_1}{(1+s)^\beta} + \frac{\mu_2}{(1+s)^{\alpha+1}} \right\}F_1(s)ds,
\end{split}
\end{equation}
where
\[
C_{f,g}:=\int_{\R^n}\left\{f(x)+g(x)\right\}\phi_1(x)dx >0.
\]
Integrating the above inequality over $[0,t]$ after a multiplication with $e^{2t}$, we get
\begin{equation}
\label{6}
\begin{split}
e^{2t}F_1(t)
\ge & \ F_1(0)+m(0)C_{f,g}\e \int_0^t e^{2s}ds \\
&+\int_0^t e^{2s}ds
\int_0^sm(r) \left\{\frac{\mu_1}{(1+r)^\beta} + \frac{\mu_2}{(1+r)^{\alpha+1}} \right\} F_1(r)dr.
\end{split}
\end{equation}
Applying a comparison argument, we have that $F_1(t)>0$ for $t>0$. Again, we should consider two cases due to the hypothesis on the data.

\textbf{Case 1: $f\geq0 (\not\equiv0)$, $g\geq0$.} In this case $F_1(0)=C_{f,0}\e>0$. The continuity of $F_1$ yields that $F_1(t)>0$ for small $t>0$. If there is the nearest zero point $t_0$ to $t=0$ of $F_1$, then \eqref{6} gives a contradiction at $t_0$.

\textbf{Case 2: $f\geq0$, $g\geq0(\not\equiv0)$.} If $f\not\equiv0$, we are in the previous case. If $f\equiv0$, then $F_1(0)=0$, $F'_1(0)= C_{0,g}\e >0$. By the continuity of $F'_1$, we have that $F'_1$ is strictly positive for small $t$, hence there exists some $t_1>0$ such that $F'_1 >0$ over $[0,t_1]$. Then $F_1$ is strictly monotone increasing on this interval, and then strictly positive on $(0,t_1]$. Now, suppose by contradiction that $t_2(>t_1)$ is the smallest zero point of $F_1$, and so $F_1>0$ on $(0, t_2)$. Then we claim that $F'_1(t_2) \leq 0$. If not, by continuity, $F'_1$ is strictly positive in a small interval $(t_3,t_2]$ for some time $t_3$ satisfying $0<t_3<t_2$. This implies that $F_1$ is strictly monotone increasing on $(t_3,t_2]$ and then negative due to the fact that $F_1(t_2)=0$, a contradiction. We then verify the claim $(F'_1(t_2) \leq 0)$. Letting $t=t_2$ in the inequality \eqref{eq:F'_1+2F_1}, noting the fact that $F_1(t_2)=0$, $F'_1(t_2) \leq 0$ and $F_1 \geq 0$ on $[0,t_2]$, we come to a contradiction. And we show that $F_1 >0$ for $t>0$ also in this case.

Therefore, coming back to \eqref{6}, we may ignore the last term, and then we have
\[
e^{2t}F_1(t) \ge F_1(0)+m(0)C_{f,g}\e\int_0^te^{2s}ds \ge \frac{1}{2}m(0)C_{f,g}\e (e^{2t}-1),
\]
from which, finally, we get the lower bound of $F_1(t)$ in the form
\begin{equation}
\label{eq:bound_F_1}
F_1(t)>\frac{1-e^{-2}}{2}m(0)C_{f,g}\e
\quad\text{for $t\ge1$}.
\end{equation}

\begin{remark}
	Note that we have to cut off the time because $f$ can vanish and so $F_1(0)$ can be equal to 0, due to our assumption on the data. If $f$ is not identically equal to zero, then the lower bound of $F_1$, i.e. \eqref{eq:bound_F_1}, holds for all $t\geq 0$.
\end{remark}
\section{Lower bound for the functional}

By H\"older inequality and using the compact support of the solution \eqref{support}, we have
\begin{equation}
\label{7}
\int_{\R^n}|u(x,t)|^pdx\ge C_2(1+t)^{-n(p-1)}|F_0(t)|^p
\quad\text{for $t\ge0$},
\end{equation}
where $C_2=C_2(n,p,R)>0$.
Plugging this inequality into \eqref{eq:bound_F'_0} and then integrating it over $[0,t]$, we have
\begin{equation}
\label{18}
F_0(t) \geq C_3\int_0^tds
\int_0^s(1+r)^{-n(p-1)}F_0(r)^pdr
\quad\text{for $t\ge0$},
\end{equation}
where $C_3:=C_2m(0)>0.$

Moreover, by H\"older inequality, Lemma \ref{lem1} and estimate \eqref{eq:bound_F_1}, we get
\begin{equation*}
\begin{split}
\int_{\R^n} |u(x,t)|^p dx
&\geq \left( \int_{\R^n} |\psi_1(x,t)|^{p/(p-1)}\right)^{1-p} |F_1(t)|^p\\
&\geq C_1^{1-p} \left(\frac{1-e^{-2}}{2}m(0)C_{f,g}\right)^p \e^p (1+t)^{(n-1)(1-p/2)} \quad\text{for $t\ge1$}.
\end{split}
\end{equation*}
Plugging this inequality into \eqref{eq:bound_F'_0} we have
\begin{equation}
\label{19}
F'_0(t)\ge C_4\e^p\int_{1}^t(1+s)^{(n-1)(1-p/2)}ds
\quad\text{for $t\ge1$},
\end{equation}
where
\[
C_4:=m(0)C_1^{1-p} \left(\frac{1-e^{-2}}{2}m(0)C_{f,g}\right)^p >0.
\]
Integrating \eqref{19} over $[1, t]$, we obtain
\begin{equation}
\label{20}
\begin{split}
F_0(t)
& \ge C_4\e^p\int_{1}^tds
\int_{1}^s(1+r)^{(n-1)(1-p/2)}dr\\
& \ge C_4\e^p(1+t)^{-(n-1)p/2}\int_{1}^tds
\int_{1}^s (r-1)^{n-1}dr\\
& =\frac{C_4}{n(n+1)}\e^p(1+t)^{-(n-1)p/2} (t-1)^{n+1}
\quad\text{for $t\ge1$}.
\end{split}
\end{equation}

\section{Iteration argument}
Now we come to the iteration argument to get the upper bound of the lifespan estimates. First we make the ansatz that $F_0(t)$ satisfies
\begin{equation}
\label{21}
F_0(t) \ge D_j(1+t)^{-a_j}(t-1)^{b_j} \quad\text{for $t\ge 1,\quad j=1,2,3,\dots$}
\end{equation}
with positive constants $D_j,a_j,b_j$,
which will be determined later.
Due to \eqref{20}, note that \eqref{21} is true when $j=1$ with
\begin{equation}
\label{22}
D_1=\frac{C_4}{n(n+1)}\e^p,
\quad a_1=(n-1)\frac{p}{2},
\quad b_1=n+1.
\end{equation}
Plugging \eqref{21} into \eqref{18}, we have
\begin{equation*}
\begin{split}
	F_0(t)
	& \ge C_3D_j^p\int_{1}^tds
	\int_{1}^s(1+r)^{-n(p-1)-pa_j}(r-1)^{pb_j}dr\\
	& \ge C_3D_j^p(1+t)^{-n(p-1)-pa_j}\int_{1}^tds
	\int_{1}^s(r-1)^{pb_j}dr\\
	& \ge \frac{C_3D_j^p}{(pb_j+2)^2}(1+t)^{-n(p-1)-pa_j}(t-1)^{pb_j+2} \quad\text{for $t\ge1$}.
\end{split}
\end{equation*}
So we can define the sequences $\{D_j\}_{j\in\N}$, $\{a_j\}_{j\in\N}$, $\{b_j\}_{j\in\N}$ by
\begin{equation}
\label{23}
D_{j+1}\ge\frac{C_3D_j^p}{(pb_j+2)^2},
\quad
a_{j+1}=pa_j+n(p-1),
\quad
b_{j+1}=pb_j+2
\end{equation}
to establish
\[
F_0(t)\ge D_{j+1}(1+t)^{-a_{j+1}}(t-1)^{b_{j+1}}\quad\text{for $t\ge1$}.
\]
It follows from \eqref{22} and \eqref{23} that for $j=1,2,3,\dots$
\begin{equation*}
a_j=p^{j-1}\left((n-1)\frac{p}{2}+n\right)-n,
\quad
b_j=p^{j-1}\left(n+1+\frac{2}{p-1}\right)-\frac{2}{p-1}.
\end{equation*}
Employing the inequality
\[
b_{j+1}=pb_j+2\le p^j\left(n+1+\frac{2}{p-1}\right)
\]
in \eqref{23}, we have
\begin{equation}
\label{24a}
D_{j+1}\ge C_5\frac{D_j^p}{p^{2j}},
\end{equation}
where
\[
C_5:=\frac{C_3}{\left(n+1+\d\frac{2}{p-1}\right)^2} >0.
\]
From \eqref{24a} it holds that
\[
\begin{split}
\log D_j&\geq p\log D_{j-1}-2(j-1)\log p+\log C_5\\
&\geq p^2\log D_{j-2}-2\big(p(j-2)+(j-1)\big)\log p+(p+1)\log C_5\\
&\geq \cdots\\
&\geq p^{j-1} \log D_1 - \sum_{k=1}^{j-1} 2 p^{k-1}(j-k) \log p + \sum_{k=1}^{j-1} p^{k-1} \log C_5\\
&= p^{j-1} \left(\log D_1-\sum_{k=1}^{j-1}\frac{2k\log p-\log C_5}{p^k}\right),
\end{split}
\]
which yields that
\[
D_j\ge\exp\left\{p^{j-1}\left(\log D_1-S_p(j)\right)\right\},
\]
where
\[
S_p(j):=\sum_{k=1}^{j-1}\frac{2k\log p-\log C_5}{p^k}.
\]
We know that $\sum_{k=0}^{\infty} x^k = 1/(1-x)$ and $\sum_{k=1}^{\infty} kx^k= x/(1-x)^2$ when $|x|<1$. Then
\[
S_p(\infty) := \lim_{j \to \infty} S_p(j) = \log\{ C_5^{p/(1-p)} p^{2p/(1-p)^2}\}.
\]
Moreover $S_p(j)$ is a sequence definitively increasing with $j$. Hence we obtain that
\[
D_j\ge\exp\left\{p^{j-1}\left(\log D_1-S_p(\infty)\right)\right\},~~~j\ge 2.
\]
Turning back to \eqref{21}, we have
\begin{equation}\label{27}
\begin{aligned}
F_0(t)\ge(1+t)^n(t-1)^{-2/(p-1)}\exp\left(p^{j-1}J(t)\right)
\quad\mbox{for}\ t\ge 1,
\end{aligned}
\end{equation}
where
\[
\begin{split}
J(t)=& -\bigg((n-1)\frac{p}{2}+n\bigg)\log(1+t)
+\bigg(n+1+\frac{2}{p-1}\bigg)\log(t-1)\\
&+\log D_1-S_p(\infty).
\end{split}
\]
For $t\geq 2$, by the definition of $J(t)$, we have
\[
\begin{split}
J(t)\geq& -\bigg((n-1)\frac{p}{2}+n\bigg)\log(2t)+\bigg(n+1+\frac{2}{p-1}\bigg)\log \bigg(\frac{t}{2}\bigg)\\
&+\log D_1-S_p(\infty)\\
=&\ \frac{\gamma(p, n)}{2(p-1)}\log t+\log D_1-\bigg((n-1)\frac{p}{2}+2n+1+\frac{2}{p-1}\bigg)\log 2-S_p(\infty)\\
=&\ \log \big(t^{\gamma(p, n)/\{2(p-1)\}}D_1\big)-C_6,\\
\end{split}
\]
where
\[
C_6:=\bigg((n-1)\frac{p}{2}+2n+1+\frac{2}{p-1}\bigg)\log 2+S_p(\infty).
\]
\par
Thus, if
\[
\begin{aligned}
t>C_7\varepsilon^{-2p(p-1)/\gamma(p, n)}
\end{aligned}
\]
with
\[
C_7:=\Big(\frac{n(n+1)e^{C_6+1}}{C_4}\Big)^{2(p-1)/\gamma(p, n)}>0,
\]
we then get $J(t)>1$,
and this in turn gives that $F_0(t)\rightarrow \infty$ by letting $j\rightarrow \infty$ in \eqref{27}. Since we assume that $t\geq2$ in the above iteration argument, we require
\[
0<\e\leq \e_0 := \left(\frac{C_7}{2} \right)^{\frac{\gamma(p,n)}{2p(p-1)}}.
\]
Therefore we get the desired upper bound,
\[
\begin{aligned}
T\leq C_7\varepsilon^{-2p(p-1)/\gamma(p, n)}
\end{aligned}
\]
for $0<\e\leq \e_0$, and hence we finish the proof of Theorem \ref{thm1}.

\section{Proof for Theorem \ref{thm2} and Theorem \ref{thm3}}

To prove the theorems in low dimensions, we proceed similarly as for Theorem \ref{thm1}, but we change the first step of the iteration argument to get the desired improvement.
\par
From \eqref{4}, using \eqref{eq:m_bound} and noting that $F_0$ is positive, we have
\[
\begin{aligned}
F_0'(t)\geq \frac{m(0)}{m(t)}F_0'(0)\geq C_8\e,
\end{aligned}
\]
where
\[
C_8:=m(0)\int_{\R^n}g(x)dx > 0
\]
due to the assumption on $g$. The above inequality implies that
\begin{equation}
\label{81}
F_0(t)\geq C_8\e t \quad\text{for $t\geq 0$}.
\end{equation}
\par
By \eqref{7} and \eqref{81}, we have
\begin{equation}
\label{81a}
\int_{\R^2}|u(x, t)|^pdx\geq C_{9}\e^p(1+t)^{-n(p-1)}t^p,
\end{equation}
with $C_{9}:=C_2C_8^p>0$.
Plugging \eqref{81a} into \eqref{eq:bound_F'_0} and integrating it over $[0, t]$ we come to
\begin{equation}
\label{83}
\begin{split}
F_0(t)&\geq m(0)C_{9}\e^p\int_0^tds\int_0^s(1+r)^{-n(p-1)}r^p dr\\
&\geq m(0)C_{9}\e^p(1+t)^{-n(p-1)}\int_0^tds\int_0^s r^pdr\\
&=C_{10}\e^p(1+t)^{-n(p-1)}t^{p+2}
\quad\text{for $t\ge0$}
\end{split}
\end{equation}
with
\[
C_{10}:=\frac{m(0)C_{9}}{(p+1)(p+2)}>0.
\]

\begin{remark}
	Note that the inequality \eqref{83} improves the lower bound of \eqref{20} for $n=2$ and $1<p<2$, and for $n=1$ and $p>1$. Hence we may establish the improved lifespan estimate as stated in Theorem \ref{thm2} and Theorem \ref{thm3}.
\end{remark}

In a similar way as in the last section, we define our iteration sequences,
$\{\widetilde{D}_j\}, \{\widetilde{a}_j\}, \{\widetilde{b}_j\}$, such that
\begin{equation}
\label{85}
\begin{aligned}
F_0(t)\geq \widetilde{D}_j(1+t)^{-\widetilde{a}_j}t^{\widetilde{b}_j}
\quad\text{for $t\geq 0$ and $j=1, 2, 3, \dots$}
\end{aligned}
\end{equation}
with positive constants, $\widetilde{D}_j, \widetilde{a}_j, \widetilde{b}_j$, and
\begin{equation}\nonumber
\begin{aligned}
\widetilde{D}_1=C_{10}\e^p, \quad \widetilde{a}_1=n(p-1), \quad \widetilde{b}_1=p+2.
\end{aligned}
\end{equation}
Combining \eqref{18} and \eqref{85}, we have
\[
\begin{split}
F_0(t)&\geq C_3\widetilde{D}_j^p\int_0^tds\int_0^s(1+r)^{-n(p-1)-p\widetilde{a}_j}r^{p\widetilde{b}_j}dr\\
&\geq \frac{C_3\widetilde{D}_j^p}{(p\widetilde{b}_j+2)^2}(1+t)^{-n(p-1)-p\widetilde{a}_j}t^{p\widetilde{b}_j+2}
\quad\text{for $t\geq 0$}.\\
\end{split}
\]
So the sequences satisfy
\begin{equation*}
\begin{aligned}
\widetilde{a}_{j+1}&=p\widetilde{a}_j+n(p-1),\\
\widetilde{b}_{j+1}&=p\widetilde{b}_j+2,\\
\widetilde{D}_{j+1}&\geq \frac{C_3\widetilde{D}_j^p}{(p\widetilde{b}_j+2)^2},
\end{aligned}
\end{equation*}
which means that
\begin{equation*}\nonumber
\begin{aligned}
\widetilde{a}_{j}&=np^j-n,\\
\widetilde{b}_{j}&=\frac{p+1}{p-1}p^{j}-\frac{2}{p-1},\\
\widetilde{D}_{j+1} &\ge C_{11} \frac{\widetilde{D}_j^p}{p^{2j}},\\
\end{aligned}
\end{equation*}
where $C_{11}:= C_3(p-1)^2/[p(p+1)]^2$, from which we get
\[
\log\widetilde{D}_{j}\geq p^{j-1}\left(\log\widetilde{D}_1-\sum_{k=1}^{j-1}\frac{2k\log p-\log C_{11}}{p^k}\right).
\]
Then proceeding as above we have
\[
\begin{aligned}
F_0(t)&\geq\widetilde{D}_j(1+t)^{n-np^j}
t^{p^j(p+1)/(p-1)-2/(p-1)}\\
&\geq (1+t)^n t^{-2/(p-1)}\exp\big(p^{j-1}\widetilde{J}(t)\big),
\end{aligned}
\]
where
\[
\widetilde{J}(t):=-np\log(1+t)+\left(p\,\frac{p+1}{p-1}\right)\log t+\log \widetilde{D}_1-\widetilde{S}_p(\infty)
\]
and
\[
\widetilde{S}_p(\infty) = \log\{ C_{11}^{p/(1-p)} p^{2p/(1-p)^2}\}.
\]
Estimating $\wt{J}(t)$ for $t\ge1$ we get
\[
\begin{split}
\wt{J}(t)
&\geq -np\log(2t)+\left(p\,\frac{p+1}{p-1}\right)\log t+\log \widetilde{D}_1-\widetilde{S}_p(\infty)\\
&=\frac{\gamma(p,n)-2}{p-1} \log t+\log \widetilde{D}_1-\widetilde{S}_p(\infty)-np\log 2,
\end{split}
\]
and then we obtain that
\[
\wt{J}(t)\ge\log\Big(t^{(\gamma(p,n)-2)/(p-1)}\widetilde{D}_1\Big)-C_{12}
\quad\text{for $t\geq 1$},
\]
where $C_{12}:=\widetilde{S}_p(\infty)+np\log 2$.
In particular,
\[
\gamma(p,n)-2=
\begin{cases}
p(3-p) &\text{if $n=2$}, \\
2p &\text{if $n=1$}.
\end{cases}
\]
By the definition of $\widetilde{D}_1$, proceeding in the same way as that in the previous section, we get the lifespan estimate in Theorem \ref{thm2} when $n=2$, and the lifespan estimate in Theorem \ref{thm3} when $n=1$.

\section{Proof for Theorem \ref{thm4}}

Let us come back to our initial equation \eqref{eq:main_problem}, with $n=p=2$. In this case we introduce another multiplier
\begin{equation}
\label{2problem_trans}
\l(t):=\exp\left(\frac{\mu_1}{2}\frac{(1+t)^{1-\beta}}{1-\beta}\right),
\end{equation}
which yields
\[
\l'(t)=\frac{\mu_1}{2(1+t)^{\beta}}\l(t)
\]
and
\[
\l''(t)=\left( \frac{\mu_1^2}{4(1+t)^{2\beta}} - \frac{\beta \mu_1}{2(1+t)^{\beta+1}}\right)\l(t).
\]
Introducing a new unknown function by
\[
w(x,t):=\l(t)u(x,t),
\]
then it is easy to get
\[
w_t=\frac{\mu_1}{2(1+t)^\beta}\l u+\l u_t
\]
and
\[
w_{tt}=\frac{\mu_1^2}{4(1+t)^{2\beta}} \l u
-\frac{\beta\mu_1}{2(1+t)^{\beta+1}} \l u
+\frac{\mu_1}{(1+t)^\beta} \l u_t+ \l u_{tt}.
\]
With this in hand the equation \eqref{eq:main_problem} can be transformed to
\begin{equation}
\label{eq:main_problem_n=p=2}
\left\{
\begin{aligned}
& w_{tt}-\Delta w = Q w + \l^{-1} |w|^2 \\
& w(x,0)=0, \quad w_t(x,0)= \l(0) \e g(x)
\end{aligned}
\right.
\end{equation}
where
\begin{equation*}
Q=Q(t):= \frac{\mu_1^2}{4(1+t)^{2\beta}} - \frac{\beta \mu_1}{2(1+t)^{\beta+1}} + \frac{\mu_2}{(1+t)^{\alpha+1}}.
\end{equation*}

A key property of the function $Q$ is its positivity. Indeed, we can write this function as $Q=\widetilde{Q}/(1+t)^{\beta+1}$, where
\begin{equation*}
\widetilde{Q}=\widetilde{Q}(t):= \frac{\mu_1^2}{4(1+t)^{\beta-1}} - \frac{\beta \mu_1}{2} + \frac{\mu_2}{(1+t)^{\alpha-\beta}},
\end{equation*}
and so it is enough to check the positivity of $\widetilde{Q}$. If $\alpha=\beta$, then $\widetilde{Q}$ is strictly decreasing to $\mu_2-\beta\mu_1/2$, that is positive by our assumption. If $\alpha<\beta$, than we can easily find the minimum $t_{0}$ of $\widetilde{Q}$, that is
\[
t_0 = -1+\left( \frac{\mu_1^2(\beta-1)}{4\mu_2(\beta-\alpha)} \right)^{\frac{1}{2\beta-\alpha-1}},
\]
and verify that the condition in \eqref{eq:cond_thm4} is equivalent to $\widetilde{Q}(t_0)\ge0$.

\begin{remark}
	Observe that:
	\begin{itemize}
		\item when $\alpha<\beta$, the condition \eqref{eq:cond_thm4} can be replaced by the more strong but easier condition
		\[
		\mu_2 \ge \frac{\mu_1^2}{4}\frac{\beta-1}{\beta-\alpha},
		\]
		that is equivalent to ask that $t_0 \le 0$, so that $\widetilde{Q}$ is increasing and positive for $t>0$;
		\item when $\alpha>\beta$, $\widetilde{Q}$ is strictly decreasing to $-\beta\mu_1/2 <0$, and then we have no chance to achieve the positivity of this function for all the time.
	\end{itemize}
\end{remark}

\begin{remark}
	We can rewrite the function $Q$ also as
	\begin{equation*}
	Q(t)=\frac{1}{4(1+t)^2}\left[\left(\frac{\mu_1}{(1+t)^{\beta-1}}-\beta\right)^2 + \frac{4\mu_2}{(1+t)^{\alpha-1}} - \beta^2\right],
	\end{equation*}
	which implies some connection with the definition \eqref{eq:def_delta} of $\delta$ in the scale invariant case ($\beta=1$) with positive mass and $\alpha=1$.
\end{remark}

Now, it is well-known that our integral equation is of the form
\begin{equation}
\label{eq:w_int}
\begin{split}
w(x,t) =&\ \frac{\l(0)\e}{2\pi} \int_{|x-y|\le t} \frac{g(y)}{\sqrt{t^2-|x-y|^2}}dy\\
&+\frac{1}{2\pi}\int_0^td\tau \int_{|x-y|\le t-\tau} \frac{Q(\tau)w(y,\tau)+\l^{-1}(\tau)|w(y,\tau)|^2}{\sqrt{(t-\tau)^2-|x-y|^2}}dy.
\end{split}
\end{equation}
Before we can move forward, we need the positivity of the solution.
\begin{lemma}\label{lem:w_pos}
	Under the assumption of Theorem \ref{thm4}, the solution $w$ of \eqref{eq:main_problem_n=p=2} is positive.
\end{lemma}
\begin{proof}
	Let $\wt{w}=\wt{w}(x,t)$ be the classical solution of the Cauchy problem
	\begin{equation*}
	\left\{
	\begin{aligned}
	& \wt{w}_{tt}-\Delta \wt{w} = Q |\wt{w}| + \l^{-1} |\wt{w}|^2, \quad \mbox{in}\ \R^n\times[0,\infty),\\
	& \wt{w}(x,0)=0, \quad \wt{w}_t(x,0)= \l(0) \e g(x), \quad x \in \R^n.
	\end{aligned}
	\right.
	\end{equation*}
	It is clear from the analogous of \eqref{eq:w_int} for $\wt{w}$ that this function is positive, and then satisfies the system \eqref{eq:main_problem_n=p=2}. But $u$ is the unique solution of \eqref{eq:main_problem}, and so $w=\lambda u$ is the unique solution of \eqref{eq:main_problem_n=p=2}. Then $w \equiv \wt{w} \ge 0$.
\end{proof}

By Lemma \ref{lem:w_pos}, we can neglect the second term on the right-hand side of \eqref{eq:w_int}. Using the relation
$|y|\le R, |x|\le t+R$
due to the support property in the first term on the right-hand side, from which the inequalities
\begin{gather*}
t-|x-y| \le t-||x|-|y|| \le t-|x|+R \quad\mbox{for}\ |x|\ge R,\\
t+|x-y|\le t+|x|+R \le 2(t+R),
\end{gather*}
we obtain that
\begin{equation*}
w(x,t) \ge \frac{\l(0)\e}{2\sqrt{2}\pi\sqrt{t+R}\sqrt{t-|x|+R}}\int_{|x-y|\le t} g(y)dy \quad\text{for $|x|\ge R$}.
\end{equation*}
If we assume $|x|+R\le t$, which implies $|x-y| \le t$ for $|y| \le R$, we get
\[
\int_{|x-y|\le t} g(y)dy = \norm{g}_{L^1(\R^2)},
\]
and then we obtain
\begin{equation}
\label{37}
w(x,t) \ge \frac{\l(0)\norm{g}_{L^1(\R^2)}}{2\sqrt{2}\pi\sqrt{t+R}\sqrt{t-|x|+R}} \e \quad\text{for $R \le |x| \le t-R$}.
\end{equation}

Defining the functional
\[
W(t) := \int_{\R^2} w(x,t)dx,
\]
we reach to
\begin{equation*}
W''(t) = Q(t)W(t) +\l^{-1}(t) \int_{\R^2} |w(x,t)|^2dx.
\end{equation*}
Noting that $W$ is also positive by Lemma \ref{lem:w_pos} (or by the fact that $W=\lambda F$), then we have
\[
W''(t) \ge \l^{-1}(t) \int_{\R^2} |w(x,t)|^2dx \ge \int_{R \le |x| \le t-R} |w(x,t)|^2dx \quad\text{for $t \ge 2R$},
\]
where we used the fact that $\l^{-1}(t) >1$. Plugging \eqref{37} into the right-hand side of the above inequality, we have
\[
W''(t) \ge \frac{\l(0)^2\norm{g}^2_{L^1(\R^2)}}{8\pi^2(t+R)}\e^2 \int_{R \le |x| \le t-R} \frac{1}{t-|x|+R}dx,
\]
which yields
\[
W''(t) \ge \frac{\l(0)^2 \norm{g}^2_{L^1(\R^2)}}{4\pi(t+R)}\e^2 \int_{R}^{t-R} \frac{r}{t-r+R}dr \quad\text{for $t \ge 2R$}.
\]
Then, the rest of the demonstration is exactly the same as that of Theorem 4.1 in \cite{T}, and we omit the details here.

\begin{remark}
	We want to emphasize that the results stated in our four Theorems are still true if we have no damping term, that is if $\mu_1=0$. In fact, a key point in our proofs was to introduce multipliers to absorb this term. If $\mu_1=0$, then $m\equiv\lambda\equiv1$ and the demonstrations proceed analogously. In this case we do not need any additional condition on $\mu_2$ in Theorem \ref{thm4}, but it is enough to ask $\mu_2>0$.
\end{remark}

\section*{Acknowledgement}
The first author is partially supported by Zhejiang Province Science Foundation (LY18A010008), NSFC (11501273, 11726612), Chinese Postdoctoral Science Foundation (2017M620128, 2018T110332), CSC(201708330548), the Scientific Research Foundation of the First-Class Discipline of Zhejiang Province (B)(201601).
The second author is partially supported by the Global Thesis study award in 2016--2017, University of Bari. And he is also grateful to Future University Hakodate for hearty hospitality during his stay
there, 12/01/2018--04/04/2018.
The third author is partially supported by the Grant-in-Aid for Scientific Research (C) (No.15K04964)
and (B)(No.18H01132),
Japan Society for the Promotion of Science, and Special Research Expenses in FY2017, General Topics
(No.B21), Future University Hakodate.
This work started when the third author was working in Future University Hakodate.
%
%
%

\end{document}